\documentclass[final,leqno]{siamltex704}
\usepackage{amssymb}
\usepackage{amsmath}
 \usepackage{mathrsfs}
 \usepackage{subfigure}
\usepackage[usenames, dvipsnames]{color}

\usepackage[colorlinks,pdfstartview=FitH,linkcolor=black]{hyperref}
\usepackage{graphicx}
\def\disp{\displaystyle}

\def\crr{\cr\noalign{\vskip1.3mm}}

\def\dref#1{(\ref{#1})}

\newtheorem{remark}{{\it Remark}}[section]

\newcommand{\R}{{\mathbb R}}

\def\A{\mathcal{A}}

\def\A{\mathcal{A}}

\begin{document}

\title{ Boundary Stabilization  and Observation  of  a  Multi-dimensional Unstable
Heat Equation
 \thanks{This work is supported  by
the National Natural Science Foundation of China (12071463,61873153,12130008).
Corresponding author: Hongyinping Feng.}
 }

\author{ Yusen Meng
\footnotemark[1] \footnotemark[2]
\ \ Hongyinping Feng
\footnotemark[3]}

\renewcommand{\thefootnote}{\fnsymbol{footnote}}
\footnotetext[1]{\scriptsize
School of Mathematical Sciences, University of Chinese Academy of Sciences, Beijing 100049, P.R. China.}
\footnotetext[2]{\scriptsize
Key Laboratory of Systems and Control, Institute of Systems Science, Academy of Mathematics
and Systems Science, Chinese Academy of Sciences, Beijing 100190, P.R. China. Email: mengyusen@amss.ac.cn.}
\footnotetext[3]{\scriptsize School of Mathematical Sciences, Shanxi University,  Taiyuan, Shanxi, 030006, P.R. China. Email: fhyp@sxu.edu.cn. }

  \slugger{sicon}{2022}{0}{0}{000--000}

\setcounter{page}{1} \maketitle

\begin{abstract}
In this paper, we consider the boundary stabilization  and observation of the  multi-dimensional unstable heat equation. Since we consider the heat equation in a  general domain,
   the usual partial differential equation backstepping method is   hard   to apply to  the considered problems.
  The unstable dynamics of the heat equation are treated by
combining   the finite-dimensional spectral truncation method and
  the dynamics compensation method.
%developed   in \cite{Feng2008,Feng2009}.
 By introducing additional finite-dimensional actuator/sensor dynamics, the unbounded stabilization/observation turns   to be a bounded one. As a result, the controller/observer design becomes more easier.
 Both the full state feedback stabilizer and the  state observer are designed.
 The exponential stability of the closed-loop and the well-posedness of the observer are obtained.

%{\color{red} Thanks to the
 % separation principle of linear system, the output feedback stabilizer is also available.}

%by improving the previous finite-dimensional spectral truncation method, we obtain exponential stability for the boundary control problem of unstable heat equations in the general domain. Next, inspired by the previous dynamics compensation method, we design the observer and demonstrate stability and well-posedness. Unlike unstable one-dimensional heat equations that can be handled with the well-known PDE backstepping method, unstable PDEs in general multidimensional domains remain a challenging problem. The assumptions of our article are valid for the heat equations in the high-dimensional general region, which basically solves the boundary control problem of the high-dimensional unstable heat equations. Then we find a suitable observer of the closed-loop system by means of the adjoint operator, and prove the fitting of the observer.

\end{abstract}



\pagestyle{myheadings} \thispagestyle{plain} \markboth{ YUSEN MENG AND HONGYINPING
FENG}{ STABILIZATION  AND OBSERVATION  OF    
HEAT EQUATION}

\section{Introduction}
When there is at least one point spectrum    in the  right-half complex plane, the system is referred to as an unstable system.
Owing to the  pole assignment theorem,
stabilization of the finite-dimensional unstable system is almost trivial.
However, the problem may become  pretty  difficult for the infinite-dimensional unstable system.
When there are only finite  point spectrums in the  right-half complex plane,
the system can be stabilized by  the finite  spectral truncation technique \cite{Triggiani1980}, \cite{Russell1978}.
 Since  heat equation  with  the  boundary or interior  source   term  is  usually an unstable system   which has  finite unstable  point spectrums,
 the finite spectral truncation technique can be  used to
 stabilize the  unstable heat  system. See, for instance,   \cite{CoronTrelat2004SICON}, \cite{Prieur2020TAC} and \cite{Prieur2020TAC2},  to name just a few.

    The partial differential equation (PDE) backstepping method is another way to stabilize  the infinite-dimensional unstable system.   It has been used to    stabilize  the  unstable heat
equations  in
    \cite{Liuw},   \cite{Smyshlyaev2005}  and \cite{WangSCL2015}. The PDE backstepping method can also cope with  other infinite-dimensional systems such as
    the first order hyperbolic equation system \cite{SmyshlyaevKrstic2008SCL},
      the
    unstable wave equation system \cite{Kristic2008} and even for  ODE-PDE cascade system \cite{AutoWang}.
     Although the PDE backstepping is powerful,
    it
      is still hard to apply to the general  multi-dimensional infinite-dimensional system.
     % such as
     %    the unstable heat equation in a general multi-dimensional region.
      Very recently,
      \cite{Feng2021} has considered the   stabilization and observation of the unstable heat equation in a general multi-dimensional region.
        By using  the dynamic compensation method   \cite{Feng2008}, \cite{Feng2009} and the      finite dimensional spectral truncation method \cite{CoronTrelat2004SICON}, both the  full state feedback
  and the state observer are proposed for the unstable multi-dimensional heat equation under the assumption that the unstable point spectrums are  algebraic simple.
     In this paper, we shall extend the results in  \cite{Feng2021}
     to the more general case.
     %More specially, we shall consider the    unstable heat equation in a general multi-dimensional domain.

 %        Inspired byand \cite{Prieur2018},  Feng proposed for the first time in \cite{Feng2021} to solve the control of high-dimensional general domain unstable heat equation by spectral method. However, it only proves that the unstable eigenvalue space is one-dimensional, which is not applicable to the general heat equation, while the eigenvalue space of the general heat equation is finite dimensional rather than one-dimensional. This paper improves this method so that the boundary control of general heat equations is established, and the boundary control of unstable heat equations is completely solved.

Let   $\Omega\subset \R^n (n \geq2)$   be  a bounded domain with $C^2$-boundary
 $\Gamma $. Suppose that  $\Gamma $  consists of two parts: $ \Gamma_0 $  and $ \Gamma_1 $, $ \Gamma_0\cup\Gamma_1=\Gamma $, with
   $ \Gamma_0 $ is a non-empty connected open set in $ \Gamma$.
 %  $ \Gamma_0=\Gamma\setminus\overline{\Gamma_1}$ and
%    $\Gamma_0\neq \emptyset$.
    Let $\nu$ be the unit
outward normal vector of $\Gamma_1$ and let $\Delta$ be the usual  Laplacian.
% which is defined by
%   \begin{equation}\label{2021162039}
% \Delta f= \sum_{i=1}^n \frac{  \partial^2 f }{  \partial x_i^2},\  \ \forall\ f\in H^2(\Omega).
% \end{equation}
We consider the following system
\begin{equation} \label{1.1}
	\begin{cases} w_{t}(x,t)=\Delta w(x,t)+\mu w(x,t)  ,\qquad (x,t)\in \Omega\times	(0,+\infty),\crr
		w(x,t)=0,\qquad (x,t)\in \Gamma_0\times(0,+\infty),\crr
		\disp \frac{\partial w(x,t)}{\partial \nu} = u(x,t), \qquad (x,t)\in \Gamma_1\times(0,+\infty),\crr
		%w(x,0)= w_0(x),  \qquad x\in \Omega,\cr
		y(x,t)=w(x,t), \qquad (x,t)\in \Gamma_1\times(0,+\infty),
	\end{cases}
\end{equation}
where $\mu>0$  is a    positive constant, $u$ is the  control and $y$ is the output.
Owing to the  source  term $\mu w(x,t)$,
 system \dref{1.1} is   unstable provided $\mu $ is large enough.
 The main goal of this paper is to design an output feedback to stabilize system \dref{1.1} exponentially. By the  separation principle
of the linear systems, the output feedback will be available once we address the following two problems: (i)
stabilize system \dref{1.1}  by a full state feedback; (ii)  design
a state observer to estimate the state online.
We will consider these two problems separately.

We
consider system (\ref{1.1}) in state space $L^2(\Omega)$. Define
\begin{equation}
	\label{1.2}
	\begin{cases} Af=\Delta f,  \ \ \forall f\in D(A)
		\disp  =\left\{  f\in H^2(\Omega)\cap H^1_{\Gamma_0}(\Omega)\ |\  \frac{\partial f}{\partial \nu}(x)=0, x\in {\Gamma_1}\right\} ,\crr
		\disp H^1_{\Gamma_0}(\Omega)=\{f\in H^1(\Omega)\ |\  f(x)=0 ,x\in{\Gamma_0}\}.
	\end{cases}
\end{equation}	
Then  $A$
   generates an exponentially
stable analytic semigroup on $L^2(\Omega)$.
It is well known (e.g. \cite[p.668]{Lasiecka2000}) that
 $D((-A)^{1/2})=H_{\Gamma_0}^1(\Omega)$ and $(-A)^{1/2}$ is a canonical
isomorphism from   $H^1_{\Gamma_0}(\Omega)$ onto  $L^2(\Omega)$.
Let $[D((-A)^{1/2} )]'=H^{-1}_{\Gamma_0}(\Omega)$  be  the   dual space  of $H^{ 1}_{\Gamma_0}(\Omega)$ with the pivot
space  $L^2(\Omega)$.
We obtain the following
Gelfand triple compact inclusions:
\begin{equation}\label{Add12264.3}
\begin{array}{l}
 D((-A)^{1/2})\hookrightarrow L^2(\Omega)
 =[L^2(\Omega)]'\hookrightarrow
[D((-A)^{1/2} )]' .
\end{array}
\end{equation}
   An extension $\tilde{A}\in
{\mathcal{L}}(H^{1}_{\Gamma_0}(\Omega), H^{-1}_{\Gamma_0}(\Omega))$ of $A$ is defined
by
 \begin{equation}\label{2020981918}
 %\begin{array}{l}
  \langle \tilde Ax,z\rangle_{H^{-1}_{\Gamma_0}(\Omega),
 H^{1}_{\Gamma_0}(\Omega)}
   =-\langle (-A)^{\frac12}x,(-A)^{\frac12}z\rangle_{L^2(\Omega)},
   \ \   \forall \; x,z\in H^{1}_{\Gamma_0}(\Omega).
 %\end{array}
 \end{equation}
  By a simple computation, the eigenpairs   $\{(\phi_j(\cdot),\lambda_j)\}_{j=1}^{\infty}$ of $A$ satisfy
 	\begin{equation}
		\label{1.1.5}
		\begin{cases} \Delta \phi_j=\lambda_j \phi_j,  \qquad x\in\Omega,  \crr
			\disp \phi_j(x)=0, x\in\Gamma_0;\ \ \
			\frac{\partial\phi_j(x)}{\partial \nu}=0, x\in\Gamma_1,
		\end{cases}\qquad  j=1,2,\cdots.
	\end{equation}

Since  the operator $A $ defined by \dref{1.2}  is self-adjoint and negative with compact resolvents, it follows from \cite[p.76, Proposition 3.2.12]{Tucsnak 2009} that the eigenvalues $\{\lambda_j\}_{j=1}^{\infty}$ are real and we can repeat each eigenvalue according to its finite multiplicity to get
\begin{equation} \label{1.1.4}
		0>\lambda_1\geq \lambda_2\geq\ldots \geq\lambda_N> \lambda_{N+1} \ldots \rightarrow -\infty.
	\end{equation}
Without loss of the generality, we assume that

\newtheorem{assumption}{Assumption}[section]
\begin{assumption}\label{assumption 1.1}
 Let the operator $A$ be given by (\ref{1.2}).
Suppose that the eigenpairs $\{(\phi_j(\cdot),\lambda_j)\}_{j=1}^{\infty}$ of $A$ satisfy  \dref{1.1.4} and
$\|\phi_j\|_{L^2(\Omega)}=1$.  Suppose that
	  there exists a constant $N>0$ such that
	\begin{equation} \label{1.1.4.1}
		\lambda_N+\mu\geq 0 \ \ \mbox{and}\ \   \lambda_{N+1}+\mu<0.
	\end{equation}

\end{assumption}

Define
the Neumann map $\Upsilon\in \mathcal{L}(L^2(\Gamma_1), $ $
H^{3/2}(\Omega))$ \cite[p.668]{Lasiecka2000} by $\Upsilon u=\psi$ if and
only if
\begin{equation}
\label{1.1.51}
\left\{\begin{array}{l}\disp \Delta \psi=0 \ \hbox{ in }\ \Omega,
\crr \disp \psi|_{\Gamma_0}=0; \qquad  \frac{\partial
\psi}{\partial\nu}\big|_{\Gamma_1}=u.
\end{array}\right.
\end{equation}
Using the Neumann map, one can write (\ref{1.1}) abstractly in $H^{-1}_{\Gamma_0}(\Omega)$ as
\begin{equation}\label{1.1.6}
	\begin{cases} w_{t}(\cdot,t)=(\tilde{A}+\mu) w(\cdot,t)+Bu(\cdot,t),\qquad t>0,\crr
		y(\cdot,t)=B^{*} w(\cdot,t), \qquad t\geq 0,
	\end{cases}
\end{equation}
where $\tilde{A}  $ is the extension of $A$ given by \dref{2020981918}, $B\in \mathcal{L}(L^2(\Gamma_1),H^{-1}_{\Gamma_0}(\Omega))$ is defined by
\begin{equation} \label{1.1.6.1}
	Bu=-\tilde{A}\Upsilon u,\qquad \forall\ u\in L^2(\Gamma_1),
\end{equation}
and
 $B^{*}$ is the adjoint of $B$, given by
 \begin{equation}\label{1.1.7}
	B^{*}f= f|_{\Gamma_1}, \qquad \forall\ f\in H^1_{\Gamma_0}(\Omega).
\end{equation}
%Noting that $(w(\cdot,t)-\psi)\in D(\tilde{A})$ and $\Delta\psi=0$, one can write (\ref{1.1})  abstractly in $H^{-1}_{\Gamma_0}(\Omega)$ as
%\begin{eqnarray}\label{mys01121539}
%	 w_{t}(\cdot,t)=&&\Delta w(\cdot,t)-\Delta\psi+\mu w(\cdot,t) \cr
%	 =&&\Delta(w(\cdot,t)-\psi)+\mu w(\cdot,t) \cr
%	 =&&\tilde{A}(w(\cdot,t)-\psi)+\mu w(\cdot,t) \cr
%	 =&&(\tilde{A}+\mu) w(\cdot,t)-\tilde{A}\Upsilon u
%\end{eqnarray}
%where $\tilde{A}  $ is the extension of $A$ given by \dref{2020981918}. If we define $B\in \mathcal{L}(L^2(\Gamma_1),H^{-1}_{\Gamma_0})$ by
%\begin{equation} \label{1.1.6.1}
%	Bu=-\tilde{A}\Upsilon u,\ \ \forall\ u\in L^2(\Gamma_1),
%\end{equation}
%then
%\begin{equation}\label{mys01121554}
%	 w_{t}(\cdot,t)=(\tilde{A}+\mu) w(\cdot,t)+Bu(\cdot,t)\qquad in \quad H_{\Gamma_0}^{-1}(\Omega)
%\end{equation}
%
%For any $f\in D(A)=D(A^{*})$ and $u\in L^2(\Gamma_1)$, it follows from \cite[p. 668]{Lasiecka2000} that the adjoint of $B$ satisfies
% \begin{equation}\label{1.1.7}
%	B^{*}f=-\gamma^{*}A f= f|_{\Gamma_1}, \qquad \forall f\in D(A)
%\end{equation}
%which,together with the denseness of $D(A)$ in $H^1_{\Gamma_0}(\Omega)$, implies that
%
% \begin{equation}\label{mys01121712}
%	B^{*}f= f|_{\Gamma_1}, \qquad \forall f\in H^1_{\Gamma_0}(\Omega)
%\end{equation}
%
%Using the operators $A$, $B$ and $B^{*}$, the control plant (\ref{1.1}) can be rewritten abstractly
%\begin{equation}\label{1.1.6}
%	\begin{cases} w_{t}(\cdot,t)=(\tilde{A}+\mu) w(\cdot,t)+Bu(\cdot,t)\qquad t>0,\cr
%		y(\cdot,t)=B^{*} w(\cdot,t), \qquad t\geq 0,
%	\end{cases}
%\end{equation}

The rest of the paper is organized as follows: In Section 2,
we give some  preliminaries for the full state feedback design.
Section 3 is devoted to the full state feedback design.  The
 exponential stability of the closed-loop system is also demonstrated in Section 3.
    Section 4 gives   some preliminaries for the  observer design
    which will be considered in
    Section 5.
       Section 6 concludes our paper and presents an outlook for future work. Some results that are less relevant to stabilizer and  observer design are arranged in  the Appendix.

\vskip0.2cm
Throughout of this paper, the space
 $\mathcal{L}(X_1, X_2)$ represents all the  bounded linear operators from the space $X_1$ to $X_2$. The space of bounded linear
operators from $X$ to itself is denoted by $\mathcal{L}(X)$. The spectrum, resolvent set and the domain of the
operator $A$ are denoted by $\sigma(A)$, $\rho(A)$ and $D(A)$, respectively. The point spectrum of $A$ is represented by $\sigma_p(A)$.
  The set of positive integers is denoted  by $\mathbb{Z}_{+}$.
  %{\color{red}$\left|P\right|$ represents the value of the determinant of $P$.}
We define inner product in $\R^N$ by
\begin{equation}\label{1.1.8}
	\langle a,b\rangle_{\R^N}=\sum_{i=1}^{N} a_i b_i,\ \ \forall\
a=(a_1,a_2,\cdots,a_N)^{\top},b=(b_1,b_2,\cdots,b_N)^{\top}\in\R^N.
\end{equation}
%where $a=(a_1,a_2,\cdots,a_N)^{\top}$, $b=(b_1,b_2,\cdots,b_N)^{\top}\in \R^N$.
%\section{Preliminaries on abstract systems }\label{Preliminary}
\section{Preliminaries for full state feedback design} \label{Preliminary}
In this section, we will give some preliminaries that are very important to the
state feedback design for system \dref{1.1}.
For any positive integer $i$, we can define, in terms of the function  $\phi_i(x)$  of (\ref{1.1.5}),   the operator $P_{\phi_i}$: $\mathbb{R}\backslash \sigma(A)  \rightarrow L^2(\Omega)$ by
\begin{equation}\label{2.1}
	P_{\phi_i}\theta=\zeta_{\phi_i}, \qquad \forall \ \theta\in \mathbb{R}\backslash \sigma(A)  , \qquad i=1,2,\cdots,N,
\end{equation}	
where $\zeta_{\phi_i}$ is the solution of following elliptic equation:
\begin{equation}\label{1.2.2}
	\begin{cases} \Delta\zeta_{\phi_i}(x)=\theta\zeta_{\phi_i}(x),\qquad x\in\Omega,\crr
		\disp \zeta_{\phi_i}(x)=0,\ \ x\in\Gamma_0 ;\qquad
		\frac{\partial\zeta_{\phi_i}(x)}{\partial \nu}=\phi_i(x) , \ \  x\in\Gamma_1.
	\end{cases}
\end{equation}

%{\color{red} We note that $\phi_i(\cdot)\neq 0$  on $\Gamma_1$ has been proved in the literature \cite[Lemma 8.1]{Feng2021}. Owing to fredholm alternative theorem,  we can conclude that \dref{1.2.2} admits a unique weak solution.}
%We note that $\phi_i(\cdot)\neq 0$  on $\Gamma_1$ has been proved in the literature \cite[Lemma 8.1]{Feng2021}. {\color{red}It is easy to know that the equation \dref{1.2.2} admits a unique solution.}

\begin{lemma}\label{lem2.1} Suppose that $\theta\in \mathbb{R}$ satisfies
	\begin{equation} \label{1.2.3}
		\theta\neq \lambda_j, \qquad  j\in \mathbb{Z}_+.
	\end{equation}
	Then the function $P_{\phi_i}$   defined  by  (\ref{2.1}) satisfies
	\begin{equation}\label{1.2.4} \langle P_{\phi_i}\theta,\phi_j\rangle_{L^2(\Omega)}=
\frac{1}{\theta-\lambda_j}\langle\phi_i,\phi_j\rangle_{L^2(\Gamma_1)}, \qquad i,j\in \mathbb{Z}_+.
	\end{equation}
\end{lemma}

\begin{proof}A straightforward computation shows that
	\begin{equation}\label{1.2.5}
		 \theta\int_{\Omega}\zeta_{\phi_i}\phi_j(x)dx
		= \int_{\Omega}\Delta\zeta_{\phi_i}\phi_j(x)dx
		= \int_{\Gamma_1}\phi_i(x)\phi_j(x)dx+\lambda_j\int_{\Omega}\zeta_{\phi_i}\phi_j(x)dx,
	\end{equation}	
	which, together with \dref{2.1}, yields \dref{1.2.4} easily.
%	\begin{equation}\label{1.2.6} \langle P_{\phi_i}\theta,\phi_j\rangle_{L^2(\Omega)}=\frac{1}{\theta-\lambda_j}\langle\phi_i,\phi_j\rangle_{L^2(\Gamma_1)}.
%	\end{equation}
\end{proof}

\vskip 0.5cm

Inspired by \cite{Feng2021}, the controller design is closely related to the
  following system:
\begin{equation}\label{1.2.7}
	\begin{cases} \disp z_{t}(x,t)=\Delta z(x,t)+\mu z(x,t)+\sum_{i=1}^N(P_{\phi_i}\theta)(x)u_i(t), \ \ (x,t)\in \Omega\times	(0,+\infty)\crr
	\disp	z(x,t)=0, \ \ (x,t)\in \Gamma_0\times(0,+\infty),\crr
		\disp \frac{\partial z (x,t)}{\partial \nu} = 0, \ \ (x,t)\in \Gamma_1\times(0,+\infty),
%		z(x,0)= z_0(x),  \ \ x\in \Omega
	\end{cases}
\end{equation}
where $P_{\phi_i}\theta$ is defined by \dref{2.1}, $\mu>0$, $N$ is a positive integer satisfies Assumption \ref{assumption 1.1} and $u_1,u_2,\cdots,u_N$ are new controls.
Since the sequence  $\{\phi_j(\cdot)\}_{j=1}^{\infty}$ under the
 Assumption \ref{assumption 1.1} forms an orthonormal basis for  $L^2(\Omega)$,  $P_{\phi_i}\theta$ and $z(\cdot,t)$ can be represented by
\begin{equation}\label{20211227735}
 P_{\phi_i}\theta=\sum_{k=1}^{\infty}f_{ki}\phi_k ,\ \ f_{ki}=\int_{\Omega}(P_{\phi_i}\theta)(x)\phi_k(x)dx
\end{equation}
and
\begin{equation} \label{20211227736}
z(\cdot,t)=\sum_{k=1}^{\infty}z_k(t)\phi_k(\cdot).
\end{equation}
%Inspired by (\cite{Coron2004}, \cite{Prieur2018}) and similarly to (\cite{Feng2008}), system (\ref{1.2.7}) can be stabilized by the finite-dimensional spectural truncation technique. Actually,
%By a simple computation, it follows that
 In view of \dref{1.2.7},  the function $z_k(t)$ in  \dref{20211227736}   satisfies
\begin{equation}\label{1.2.8}
    \begin{aligned}
	 \dot{z}_k (t)
	=&\int_{\Omega}z_t(x,t)\phi_k(x)dx  \crr
	=&\int_{\Omega}\left[\Delta z(x,t)+\mu z(x,t)+\sum_{i=1}^{N}(P_{\phi_i}\theta)(x)u_i(t) \right] \phi_k(x)dx \crr
	=&(\lambda_k+\mu)z_k(t)+\sum_{i=1}^N f_{ki}u_i(t),\ \ \ k=1,2,\cdots.
     \end{aligned}
\end{equation}
Since $\lambda_k+\mu<0$ provided $k> N$,
  $z_k(t)$ is stable for all $ k>N$. Consequently,  it is   sufficient to consider $z_k(t)$ for $k\leq N$, which satisfy the following finite-dimensional system:
\begin{equation}\label{1.2.9}
	\dot{Z}_{N}(t)=\Lambda_{N} Z_{N}(t)+F_{N}u(t),\ \ Z_{N}(t)=\begin{pmatrix}
 z_1(t)\\ z_2(t)\\ \vdots\\z_{N}(t)\end{pmatrix},\ \
 u(t)=\begin{pmatrix}
 u_1(t)\\ u_2(t)\\ \vdots\\u_{N}(t)\end{pmatrix},
\end{equation}
where    $\Lambda_{N}$ and  $F_{N}$ are defined by
\begin{equation}\label{2.10}
	\begin{cases} \Lambda_{N}={\rm diag}(\lambda_1+\mu,\cdots,\lambda_N+\mu),\crr
		F_{N}=
		\begin{pmatrix}
			f_{11} & f_{12}  & \cdots & f_{1N}  \\
			f_{21} & f_{22}  & \cdots & f_{2N}  \\
			\vdots & \vdots  & \cdots & \vdots  \\
			f_{N1} & f_{N2}  & \cdots & f_{NN}  \\
		\end{pmatrix} _{N\times N} ,\crr
\disp f_{ki}=\int_{\Omega}(P_{\phi_i}\theta)(x)\phi_k(x)dx,\qquad i,k=1,2,\cdots,N.
	\end{cases}
\end{equation}

\begin{lemma}\label{lem1.2.2}
In addition to  Assumption \ref{assumption 1.1},  assume that $\theta\in \mathbb{R}$   satisfies  (\ref{1.2.3}). Then, there exists a matrix $L_{N}=(l_{ij})_{N\times N} $ such that $\Lambda_{N}+F_{N}L_{N}$ is Hurwitz, where $\Lambda_{N}$ and $F_{N}$ are defined by (\ref{2.10}). Moreover, the operator $ A+\mu+\sum_{i=1}^{N}(P_{\phi_i}\theta)K_i$ generates an exponentially stable $C_0$-semigroup on $L^2(\Omega)$, where $P_{\phi_i}\theta$  is given by \dref{2.1} and $K_i$ is given by
	
\begin{equation} \label{2.12}
	K_i: g\rightarrow \int_{\Omega}g(x)\sum_{k=1}^{N} l_{ik}\phi_k(x)dx, \ \ \forall \ g\in L^2(\Omega) ,\ \ \forall \ i=1, 2 \cdots, N.
	\end{equation}
\end{lemma}

\begin{proof}
Since the sequence  $\{\phi_j(\cdot)\}_{j=1}^{\infty}$ under the
 Assumption \ref{assumption 1.1} are  linearly
independent on  $L^2(\Omega)$,
it follows from  Lemmas
\ref{lem6.1.2} and	\ref{lem6.x} in Appendix   and \dref{1.2.4} that  the pair  $(\Lambda_{N}, F_{N})$ is controllable.
 Hence, there exists a matrix $L_{N}$ such  that $\Lambda_{N}+F_{N}L_{N}$ is Hurwitz. Since  $A+\mu$ generates an analytic semigroup $e^{(A+\mu)t}$ on $L^2(\Omega)$ and $ \sum_{i=1}^{N} P_{\phi_i}\theta K_i\in \mathcal{L}(L^2(\Omega))$, it follows from \cite[Corollary 2.2, p.81]{Pazy1983}  that  $  A+\mu+\sum_{i=1}^{N} P_{\phi_i}\theta K_i$   generates an analytic semigroup on $L^2(\Omega)$ as well.
By a straightforward computation,  $A+\mu+\sum_{i=1}^{N} P_{\phi_i}\theta K_i$
is the inverse of a compact operator. Thanks to \cite[Theorem 4.3, p.118]{Pazy1983}, the proof will be
accomplished if we can show that the point spectrum of
$ A+\mu+\sum_{i=1}^{N} P_{\phi_i}\theta K_i$ satisfies
\begin{equation}  \label{2.13}
	\sigma _p(A+\mu+\sum_{i=1}^{N} P_{\phi_i}\theta K_i)\subset\{s\ |\ {\rm Re}\, s<0\}.
\end{equation}

Actually, for all $ \lambda\in \sigma_p (A+\mu+\sum_{i=1}^{N} P_{\phi_i}\theta K_i)$, we  consider following characteristic equation $(A+\mu+\sum_{i=1}^{N} P_{\phi_i}\theta K_i)g=\lambda g$
with $g\neq 0$.  Since
the sequence $\{\phi_j(\cdot)\}_{j=1}^{\infty}$  forms an orthonormal basis for  $L^2(\Omega)$,  we can suppose that
\begin{equation}  \label{FH2022151123}
	0\neq g=\sum_{k=1}^{\infty}g_k\phi_k,\ \ g_k=\langle g,\phi_k\rangle_{L^2(\Omega)},\quad k=1,2,\cdots.
\end{equation}
As a result, the characteristic equation becomes
\begin{equation}\label{FH20222151137}
\begin{aligned}
	 \sum_{k=1}^{\infty}\lambda g_k\phi_k =&\sum_{k=1}^{\infty}g_k(A+\mu)\phi_k
         +\sum_{i =1 }^{N}\sum_{k=1}^{\infty}(P_{\phi_i}\theta) g_k K_i\phi_k  \crr
	=&\sum_{k=1}^{\infty}g_k(\lambda_k+\mu)\phi_k+ \sum_{i,k =1 }^{N}(P_{\phi_i}\theta) l_{ik}g_k.
\end{aligned}
\end{equation}

When $(g_1,g_2,\cdots,g_N)\neq0$,
%we have      $g_j=0$ for $j=N+1,N+2,\cdots$. Hence,  the characteristic equation becomes
%\begin{eqnarray}\label{2.14}
%	 \sum_{j=1}^{N}\lambda g_j\phi_j
%	=&&\sum_{j=1}^{N}g_j(A+\mu)\phi_j+\sum_{i,k=1}^{N}(P_{\phi_i}\theta) K_ig_k\phi_k  \cr
%	=&&\sum_{j=1}^{N}g_j(\lambda_j+\mu)\phi_j+\sum_{i,k=1}^{N}(P_{\phi_i}\theta) l_{ik}g_k.
%\end{eqnarray}
we take the inner product with $\phi_j$, $ j = 1,2,\cdots,N$ on equation
(\ref{FH20222151137}) to obtain
\begin{equation}\label{2.15}	
	\lambda g_j=g_j(\lambda_j+\mu)+\sum_{i,k=1}^{N}f_{ji} l_{ik}g_k,
\end{equation}
which, together with (\ref{2.10}), leads to
\begin{equation} \label{2.15.1}(\lambda-\Lambda_N-F_NL_N)(g_1,g_2,\cdots,g_N)^{\top}= 0.\end{equation}
Since $(g_1,g_2,\cdots,g_N)^{\top}\neq 0$, \dref{2.15.1} implies that
\begin{equation}\label{2.16}
	{\rm Det}(\lambda-\Lambda_N-F_NL_N)= 0.
\end{equation}
Hence, $\lambda\in \sigma(\Lambda_N+F_NL_N)\subset \{s\ |\ {\rm Re} \, s<0\}$ due to   that  $\Lambda_N+F_NL_N$ is Hurwitz.

\vskip 0.3cm

When $(g_1,g_2,\cdots,g_N)=0$,  there exists a   $j_0>N$  such that $\int_{\Omega}g(x)\phi_{j_0}(x)dx\neq 0$ due to $g\neq 0$, and at the same time, (\ref{FH20222151137}) is reduced to
\begin{equation}\label{meng202201051328}
 \sum_{k=1}^{\infty}\lambda g_k\phi_k=\sum_{k=1}^{\infty}g_k(\lambda_k+\mu)\phi_k.
 \end{equation}
Take the inner product
with $\phi_{j_0}$ on equation (\ref{meng202201051328}) to get
\begin{eqnarray}\label{2.17}
 (\lambda_{j_0}+\mu)g_{j_0}=\lambda g_{j_0},
\end{eqnarray}
 which, together with (\ref{1.1.4.1}), implies that $\lambda= \lambda_{j_0}+\mu<0$. So  ${\rm Re} \, \lambda<0$. Therefore, we can get  \dref{2.13}.
\end{proof}

\section{State feedback} \label{shiftsemigroup}

	This section is devoted to the full state feedback design for  system (\ref{1.1}).
To this end, we first define, in terms of the eigenfunctions $\phi_1,\phi_2,\cdots,\phi_N$ that are given by \dref{1.1.5}, the operator  $B_{v}\in \mathcal{L}\big(\R^N , L^2(\Gamma_1) \big)$ by following equation
\begin{equation} \label{Fh2022151515}
B_{v} c =
%\begin{pmatrix}
%       B_{1}c\\
%       B_{2}c\\
%       \vdots\\
%       B_Nc
%     \end{pmatrix}=
\sum_{j=1}^{N}c_j\phi_j(x)
 ,\quad x\in \Gamma_1,\ \
 \ \ \forall\  c=\begin{pmatrix}
         c_1,
         c_2 ,
       \cdots,
        c_N
     \end{pmatrix}^{\top}\in \mathbb{R}^N .
\end{equation}
 Inspired by \cite{Feng2008}, we consider the following dynamics feedback
\begin{equation}
	\begin{cases}
		\label{3.1}
		\disp u(x,t)= v(x,t) ,\qquad x\in \Gamma_1, \crr
		\disp {v}_t(\cdot,t)=-\alpha I v (\cdot,t)+B_{v}u_v(t) \quad \mbox{in} \quad L^2(\Gamma_1) ,
	\end{cases}
\end{equation}
where
%$Q: [L^2(\Gamma_1)]^N\to  L^2(\Gamma_1) $ is given by
% \begin{equation}
%	Qf=\sum_{j=1}^{N}f_j ,\ \ \forall\ f= \begin{pmatrix}
%         f_1  \\
%         f_2  \\
%       \vdots\\
%        f_N
%     \end{pmatrix}\in [L^2(\Gamma_1)]^N,
%\end{equation}
 $\alpha>0$ is a tuning parameter,  $I $ is the  identity operator  in $ L^2(\Gamma_1) $ and
 \begin{equation}\label{2022151709}
  	u_v(t)=\begin{pmatrix}
         u_1(t) ,
         u_2 (t) ,
       \cdots,
        u_N(t)
     \end{pmatrix}^{\top} \in\R^N,	
\end{equation}
  is  a new    control to be designed later.
%\begin{equation}
%	\begin{cases}
%		\label{3.1}
%		\disp u(x,t)=\sum_{j=1}^N v_j(x,t) ,\qquad x\in \Gamma_1, \crr
%		\disp v_{it}(\cdot,t)=-\alpha v_{i}(\cdot,t)+\sum_{j=1}^{N}B_{v_j}u_{v_j}(t) \quad \mbox{in} \quad L^2(\Gamma_1) \qquad i=1,2\cdots,N,
%	\end{cases}
%\end{equation}
% where $\alpha>0$ is a tuning parameter,  $N$ is a positive integer satisfies Assumption \ref{assumption 1.1},  $u_{v_i}(t)$ is a new scalar control to be designed and the operator $B_{v_i}\in \mathcal{L}(\mathbb{R},L^2(\Gamma_1))$ is given by
%\begin{equation} \label{3.2}
%B_{v_i} c=c \phi_i(x)|_{\Gamma_1},\ \ \forall c\in \mathbb{R},\ \ x\in \Gamma_1, \ \ i=1,2\cdots,N
%\end{equation}
%where the function  $\phi_i(x)$ is give by  (\ref{1.1.5}).
 Under the controller (\ref{3.1}), the control plant
 %(\ref{1.1.6}), or equivalently
 (\ref{1.1})  becomes
\begin{equation}
	\begin{cases}
		\label{2.1.5}
		\disp w_t(\cdot,t)=(\tilde{A}+\mu)w(\cdot,t)+	B v(\cdot,t) \quad \mbox{in}\  \ H^{-1}_{\Gamma_0}(\Omega), \crr
		\disp v_{ t}(\cdot,t)=-\alpha Iv (\cdot, t)+ B_{v}u_{v}(t) \quad \mbox{in}\  \ L^2(\Gamma_1) ,
	\end{cases}	
\end{equation}
where $\tilde{A}  $ is the extension of $A$ given by \dref{2020981918} and $B\in \mathcal{L}\big(L^2(\Gamma_1),H^{-1}(\Omega)\big)$ is given  by
\dref{1.1.6.1}.

Since (\ref{2.1.5})  is a cascade system, the ``$v$-system" can be regarded as the actuator dynamics of the control plant ``$w$-system". Therefore, we can use the actuator dynamics compensation method proposed in \cite{Feng2008} to stabilize the system \dref{2.1.5}.
Actually, the controller    \dref{2022151709} can be designed as
\begin{equation}\label{FH2.1.6}
u_{i }( t) =
    -[      K_iw(\cdot, t) +  K_iS v (\cdot,t)] ,\quad i=1,2,\cdots,N,
\end{equation}
%\begin{equation}\label{FH2.1.6}
%u_{v }( t) =
%    -\begin{pmatrix}
%         K_1w(\cdot, t) +  K_1S v (\cdot,t)  \crr
%          K_2w(\cdot, t) +  K_2S v (\cdot,t)\crr
%       \vdots\\
%          K_Nw(\cdot, t) +  K_NS v (\cdot,t)
%     \end{pmatrix},
%\end{equation}
%\begin{equation}\label{2.1.6}
%u_{v_i}(\cdot,t) =\frac{1}{N} K_iw(\cdot, t) +\frac{1}{N} K_iS(v_1(\cdot,t)+\cdots+v_N(\cdot,t))  \qquad \forall i=1,2,\cdots,N
%\end{equation}
where $K_i$ are given by  (\ref{2.12}),
%such that $A+\mu+\sum_{i=1}^{N}P_{\phi_i}\theta K_i$ generates an exponential stable $C_0$-semigroup on $L^2(\Omega)$,
$S\in \mathcal{L}\big(L^2(\Gamma_1),L^2(\Omega)\big)$ solves the Sylvester equation
\begin{equation}\label{2.1.10}
	(\tilde{A}+\mu)S+\alpha S=B.
\end{equation}
Combining  (\ref{2.12}) and  \dref{2.1.12},      the controller (\ref{FH2.1.6}) turns to be
\begin{equation}\label{mys2.1.17}
	\disp u_{i}( t) =
	 -  \int_{\Omega}\big[w(x,t)-\varphi_v(x,t)\big] \left[\sum_{k=1}^N l_{ik}\phi_k(x)\right]dx  ,
\end{equation}
%\begin{equation}\label{2.1.17}
%u_{i}(t) =-\frac{1}{N} \int_{\Omega}\big(w(x,t)-\varphi_v(x,t)\big)\cdot\big(\sum_{k=1}^N l_{ik}\phi_k(x)\big)dx  \qquad i=1,2,\cdots,N
%\end{equation}
where $\varphi_v $ satisfies following equation
\begin{equation}
	\begin{cases} \label{2.1.18}
		\Delta \varphi_v(\cdot,t)=(-\alpha-\mu)\varphi_v(\cdot,t) \qquad {\mbox{in}} \quad\Omega, \crr
	\disp	\varphi_v(x,t) =0 , \qquad  x\in\Gamma_0, \crr
		\disp \frac{\partial \varphi_v(x,t) }{\partial \nu}=v(x,t), \qquad x\in\Gamma_1.
	\end{cases}
\end{equation}
Combining \dref{1.1}, \dref{Fh2022151515}, \dref{3.1}, \dref{2022151709} and \dref{mys2.1.17}, we obtain the closed-loop  system \begin{equation}
	\begin{cases}
		\label{2.1.20}
		w_t(x,t)=\Delta w(x,t)+	\mu w(x,t),  \qquad (x,t)\in \Omega\times(0,+\infty) ,\crr
		w(x,t)=0,\qquad (x,t)\in \Gamma_0\times(0,+\infty),\crr
		\disp \frac{\partial w(x,t)}{\partial \nu} = v(x,t), \qquad (x,t)\in \Gamma_1\times(0,+\infty),\crr
		\disp v_{t}(x,t)=-\alpha v(x, t)-\sum_{j=1}^{N}\phi_j(x)
		\disp  \int_{\Omega}\big[w(x,t)
		-\varphi_v(x,t)\big]\left[\sum_{i=1}^N l_{ji}\phi_i(x)\right]dx, \crr
		\hspace{5cm} (x,t)\in \Gamma_1\times(0,+\infty),\cr
		\Delta \varphi_v(\cdot,t)=(-\alpha-\mu)\varphi_v(\cdot,t)  \quad \hbox{in} \ \ \Omega, \crr
		\varphi_v(x,t)=0 ,\qquad  x\in\Gamma_0, \crr
		\disp \frac{\partial \varphi_v(x,t)}{\partial \nu}=v(x,t), \qquad x\in\Gamma_1.
	\end{cases}	
\end{equation}

\vskip 0.5cm
\begin{lemma}\label{lem.2}  Let the operators  $\tilde{A}$ and $B$ be  given by (\ref{2020981918}) and (\ref{1.1.6.1}), respectively.
Suppose that   $\alpha$ satisfies
	\begin{equation} \label{2.1.11}
	\alpha+\mu \in\rho(-A).
	\end{equation}
	Then the solution of Sylvester equation (\ref{2.1.10}) satisfies 		
	\begin{equation} \label{2.1.12}
	Sg=-\varphi_g\in L^2(\Omega), \qquad  \forall \ g\in L^2(\Gamma_1),
	\end{equation}	
	where $\varphi_g$ satisfies  following equation
	\begin{equation}
		\begin{cases} \label{2.1.13}
			\Delta \varphi_g(x)=(-\alpha-\mu)\varphi_g(x), \quad x\in\Omega ,\crr
		\disp	\varphi_g(x)=0,\ x\in\Gamma_0;\ \
			\frac{\partial \varphi_g(x)}{\partial \nu}=g(x), \ x\in \Gamma_1.
		\end{cases}
	\end{equation}
 	Moreover,  for any $c=(c_1,c_2,\cdots,c_N)^{\top}\in \mathbb{R}^N$
	\begin{equation} \label{2.1.13.1}
	SB_{v} c=-\sum_{i=1}^N  c_i P_{\phi_i}\theta,  \qquad  \theta=-\alpha-\mu,
	\end{equation}	
	where $B_v$ is  given by \dref{Fh2022151515}
and
 $P_{\phi_i}\theta$ are given by (\ref{2.1}).
\end{lemma}

\vskip 0.5cm

\begin{proof}		
By \dref{2.1.10} and  (\ref{2.1.11}), we  get
\begin{equation} \label{2.1.14}
S=(\alpha+\mu+\tilde{A})^{-1}B.
\end{equation}
It follows from \dref {1.1.51},  \dref{1.1.6.1} and \dref{2.1.13}  that
\begin{equation}\label{2.1.15}
\begin{aligned}
	 \disp (\alpha+\mu+\tilde{A})\varphi_g
	= (\alpha+\mu+\tilde{A})\varphi_g-\tilde{A}\Upsilon g+\tilde{A}\Upsilon  g \crr
	\disp =  (\alpha+\mu)\varphi_g+\tilde{A}(\varphi_g-\Upsilon  g)+\tilde{A}\Upsilon  g
	=  \tilde{A}\Upsilon  g
	= -Bg,
\end{aligned}
\end{equation}
which together with (\ref{2.1.14}), leads to $Sg=-\varphi_g$. Consequently, we combine  (\ref{2.1}) and (\ref{Fh2022151515}) to get \dref{2.1.13.1}.
%\begin{equation} \label{2.1.16}
%	SB_{v} c=-\sum_{i=1}^N c_i P_{\phi_i}\theta, \qquad  \theta=-\alpha-\mu.
%\end{equation}
\end{proof}

%
%By (\ref{Fh2022151515}), (\ref{2.1.5}) and (\ref{mys2.1.17}), we obtain the closed-loop system
%\begin{equation}
%	\begin{cases}
%		\label{2.1.19}
%		\disp w_t(\cdot,t)=(\tilde{A}+\mu)w(\cdot,t)+B v(\cdot,t),  \qquad in \ \ \Omega \cr
%		\disp v_{t}(\cdot,t)=-\alpha v(\cdot, t)-\sum_{j=1}^{N}\phi_j(\cdot) \cr
%		\disp \qquad \qquad  \times\int_{\Omega}\big(w(x,t)-\varphi_v(x,t)\big)\cdot\big(\sum_{i=1}^N l_{ji}\phi_i(x)\big)dx, \qquad on \ \ \Gamma_1 \cr
%		\Delta \varphi_v(\cdot,t)=(-\alpha-\mu)\varphi_v(\cdot,t) \qquad  in\quad\Omega \cr
%		\varphi_v(x,t)=0 \qquad x\in \Gamma_0 \cr
%		\frac{\partial \varphi_v(x,t)}{\partial \nu}=v(x,t) \qquad x\in \Gamma_1
%	\end{cases}	
%\end{equation}

\vskip0.2cm
\begin{theorem}\label{thm2.2}
	In addition to Assumption \ref{assumption 1.1}, suppose that  $\alpha>0$ satisfies
	\begin{equation}
	\alpha+\mu+\lambda_j\neq 0, \qquad j\in \mathbb{Z}_+.
	\end{equation}
	Then, there exists a matrix $L_N=(l_{ij})_{N\times N}  $ such that $\Lambda_N+F_NL_N $ is Hurwitz, where $\Lambda_N$ and $F_N$ are defined by (\ref{2.10}). Moreover, for any initial value $\big(w(\cdot,0),v(\cdot,0)\big)^{\top}\in L^2(\Omega)\times L^2(\Gamma_1)$, system (\ref{2.1.20}) admits a unique solution $(w,v)^{\top} \in C\big([0,\infty);L^2(\Omega)\times L^2(\Gamma_1)\big)$ that decays to zero exponentially in $L^2(\Omega)\times L^2(\Gamma_1)$ as $t\rightarrow +\infty$. %Furthermore, if the domain of the initial value satisfies $ D(A)\times L^2(\Gamma_1)$, the solution %$(w,v)^{\top}\in C^1\big([0,\infty);L^2(\Omega)\times L^2(\Gamma_1)\big)$ is classical.
\end{theorem}

\begin{proof}
 We combine the (\ref{2022151709}) and \dref{FH2.1.6} to get
\begin{equation}\label{mys01061139}
u_{v}( t)=-Kw(\cdot,t)-K S v(\cdot,t),
\end{equation}
 where $K\in \mathcal{L}(L^2(\Omega),\R^N)$ is defined by
\begin{equation}\label{mys01061131}
Kg=(K_1 g,K_2 g,\cdots,K_N g)^{\top},\qquad \forall g\in L^2(\Omega).
\end{equation}
 In view of the operators given by  \dref{1.2}, \dref{1.1.6.1}, \dref{Fh2022151515},  \dref{2.1.12} and \dref{mys01061131}, the closed-loop system \dref{2.1.20} can be written as the abstract form:
\begin{equation}\label{2.1.21}
\frac {d}{dt}\big(w(\cdot,t),v(\cdot,t)\big)^{\top}=\mathcal{A}\big(w(\cdot,t),v(\cdot,t)\big)^{\top},
\end{equation}
where the operator $\mathcal{A}$: $D(\mathcal{A})\subset L^2(\Omega)\times L^2(\Gamma_1) \rightarrow  L^2(\Omega)\times L^2(\Gamma_1)$ is defined by

\begin{equation} \label{2.1.22}
\mathcal{A}=
\begin{pmatrix}
	A+\mu  & B  \\
	-B_vK  & -B_vKS-\alpha I  \\
\end{pmatrix},\ \ D(\A)=D(A)\times L^2(\Gamma_1).
\end{equation}
%with $D(\A)=D(A)\times L^2(\Gamma_1)$.
It is sufficient to prove
that the operator $\mathcal{A}$ generates an exponentially stable $C_0$-semigroup on $L^2(\Omega)\times L^2(\Gamma_1)$.

Inspired by  \cite{Feng2008}, we introduce  following transformation
\begin{equation} \label{2.1.23}
\mathbb{S}(f,g)^{\top}=(f+S g,g)^{\top}, \qquad (f,g)^{\top}\in L^2(\Omega)\times L^2(\Gamma_1),
\end{equation}
where $S\in\mathcal{L} (L^2(\Gamma_1), L^2(\Omega))$ solves the Sylvester equation 	(\ref{2.1.10}).
By a simple computation,  $\mathbb{S}\in L^2(\Omega)\times L^2(\Gamma_1)$ is invertible and its inverse is
\begin{equation}
\mathbb{S}^{-1}(f,g)^{\top}=(f-S g,g)^{\top}, \qquad (f,g)^{\top}\in L^2(\Omega)\times L^2(\Gamma_1).
\end{equation}
Moreover,
\begin{equation}\label{2.xx}
\mathbb{S}\mathcal{A}\mathbb{S}^{-1}=\mathcal{A}_{\mathbb{S}}, \qquad D(\mathcal{A}_{\mathbb{S}})=\mathbb{S}D(\mathcal{A}),
\end{equation}
where $\mathcal{A}_{\mathbb{S}}$ satisfies
\begin{equation} \label{2.1.24}
\mathcal{A}_{\mathbb{S}}=
\begin{pmatrix}
	A+\mu-SB_vK  & 0  \\
	B_vK  & -\alpha  \\
\end{pmatrix}.
\end{equation}
Here  $S B_v$ is given by (\ref{2.1.13.1}) and $K$ is given by (\ref{mys01061131}).  According to the Lemma \ref{lem1.2.2},  the operator $A+\mu-S B_vK=A+\mu+\sum_{i=1}^N P_{\phi_i}\theta K_i$ generates an exponentially stable $C_0$-semigroup on $L^2(\Omega)$ with $\theta=-\alpha-\mu$. Owing  to the block-triangle structure and \cite[Lemma 5.1]{Weiss1997TAC}, the operator $\mathcal{A}_{\mathbb{S}}$ generates an exponentially stable $C_0$-semigroup $e^{\mathcal{A}_{\mathbb{S}}t}$ on $L^2(\Omega)\times L^2(\Gamma_1)$. Therefore, the operator $\mathcal{A}$ also generates an exponentially stable $C_0$-semigroup on $L^2(\Omega)\times L^2(\Gamma_1)$ due to the similarity (\ref{2.xx}).
\end{proof}
\vskip0.2cm
\begin{remark} \label{ReFH}
We point out that   Theorems \ref{thm2.2} is better than
the results in  \cite{Barbu2012} and \cite{Barbu2013}, where
the additional assumption that
  the eigenfunctions $\phi_j, \ j\leq N$ are independent on $L^2(\Gamma_1)$ must be required.
  %Our assumption ia easy to hold.
  %It follows from \cite[P77, 3.2.13]{Tucsnak 2009}, for  self-adjoint operator $A$ with compact resolvents, we can always find $N>0$ satisfying \dref{1.1.4.1}.
\end{remark}
\section{Preliminaries for observer design}

 This section is devoted to the preliminaries on the observer design
 that is closely related to  the adjoint of the  operator  $\mathcal{A}$   given by \dref{2.1.22}.
 We first  compute the adjoint operators of the $A$, $B $, $B_v$, $K$ and $S$.
 Since the adjoint  of $B$ has been  given  by \dref{1.1.7} and
   $A^*=A$,
 we  only need to compute  the  adjoint operators of $B_v$, $K$ and $S$.
  %which are defined by (\ref{Fh2022151515}), (\ref{mys01061131})  and \dref{2.1.12}, respectively.

By  (\ref{Fh2022151515}),   the adjoint operator $B_v^{*}\in \mathcal{L} \big(L^2(\Gamma_1),\mathbb{R}^N\big)$ of $B_v$ satisfies
\begin{equation}
\label{mys01091645}
\begin{array}{ll}
	\disp \langle c,B_v^*g \rangle_{\R^N}&\disp = \langle B_v c,g \rangle_{L^2(\Gamma_1)}
     = \int_{\Gamma_1}\sum_{j=1}^{N}c_j\phi_j(x)g(x)dx \crr
	 &\disp=\sum_{j=1}^{N}c_j\int_{\Gamma_1}\phi_j(x)g(x)dx
\end{array}
\end{equation}
for all $c=(c_1,c_2,\cdots,c_N)^{\top}\in \R^N$ and $g\in L^2(\Gamma_1)$. As a result,
\begin{equation}\label{mys01091933}
	B_v^*g=\left(\int_{\Gamma_1}\phi_1(x)g(x)dx,\cdots,\int_{\Gamma_1}\phi_N(x)g(x)dx \right)^{\top}, \quad \forall \ g\in L^2(\Gamma_1).
\end{equation}
Similarly, it follows from (\ref{mys01061131})
  that   $K^*\in \mathcal{L}\big(\R^N,L^2(\Omega)\big)$    satisfies
\begin{equation} \label{mys01091727}
\begin{array}{ll}
	 \langle f,K^{*} c \rangle_{L^2(\Omega)} &\disp =\langle Kf,c \rangle_{\R^N}
	= \sum_{i=1}^N c_i \int_{\Omega}f(x) \sum_{j=1}^{N}l_{ij}\phi_j(x)dx\crr
	&\disp =\int_{\Omega}f(x) \sum_{i,j=1}^{N}c_i l_{ij}\phi_j(x)dx
\end{array}
\end{equation}
for all $c=(c_1,c_2,\cdots,c_N)^{\top}\in \R^N$ and $f\in L^2(\Omega)$.  \dref{mys01091727} implies  that
\begin{equation} \label{mys01091956}
	 K^*c= \sum_{i,j=1}^{N}c_i l_{ij}\phi_j(x), \quad \forall \ c=(c_1,c_2,\cdots,c_N)^{\top}\in \R^N.
\end{equation}

  To   compute $S^*$, we suppose that $\xi_f(x)$ is the solution of the following elliptic equation
\begin{equation}\label{mys01091845}
	\begin{cases}
		\disp \Delta \xi_f(x) =-(\alpha+\mu)\xi_f(x) +f(x)    \ \ x \in \Omega, \crr
		\disp \xi_f(x)=0 ,\ x\in \Gamma_0;\ \ \frac{\partial \xi_f(x)}{\partial \nu} =0,\ x\in \Gamma_1,
	\end{cases}	
\end{equation}
where $\alpha$ and $\mu $ satisfy  (\ref{2.1.11}).
Owing to  Fredholm alternative theorem,  equation \dref{mys01091845}
 admits a unique solution  $\xi_f$ for each inhomogeneous term $f\in L^2(\Omega)$. So the   function $\xi_f(x)$  makes sense.
  In view of  (\ref{2.1.12}),
  for any $g\in L^2(\Gamma_1)$ and $f\in L^2(\Omega)$,
   the adjoint operator $S^*\in \mathcal{L}\big(L^2(\Omega),L^2(\Gamma_1)\big)$ satisfies
\begin{equation} \label{mys01091850}
\begin{array}{ll}
	\langle g,S^{*}f\rangle_{L^2(\Gamma_1)}&\disp=\langle Sg,f\rangle_{L^2(\Omega)}
	= \int_{\Omega}-\varphi_g(x) f(x) dx\crr
	&\disp=\int_{\Omega}-\varphi_g(x)(\Delta \xi_f(x)+(\alpha+\mu)\xi_f(x))dx\crr
	&\disp=\int_{\Omega}(\Delta \varphi_g(x) \xi_f(x)-\varphi_g(x) \Delta \xi_f(x))dx\crr
	&\disp=\int_{\Gamma}\left(\frac{\partial \varphi_g(x)}{\partial \nu}\xi_f(x)-\frac{\partial \xi_f(x)}{\partial \nu}\varphi_g(x)\right)dx\crr
	&\disp=\int_{\Gamma_1}g(x)\xi_f (x) dx ,
\end{array}
\end{equation}
 which yields
\begin{equation}\label{mys01091957}
	S^* f= \xi_f, \qquad \forall \ f\in L^2(\Omega),
\end{equation}
 where $\xi_f$ satisfies (\ref{mys01091845}).

With the operators $B^*, A^*  $,  $B_v^*$, $K^*$ and $S^*$ at hand, a simple computation shows that
the operator $\mathcal{A}^{*}: D(\mathcal{A}^{*})\subset L^2(\Omega)\times L^2(\Gamma_1)\rightarrow L^2(\Omega)\times L^2(\Gamma_1)$, the adjoint of the  operator  $\mathcal{A}$   given by \dref{2.1.22},   is
 \begin{equation} \label{mys01092134}
	\disp \mathcal{A}^{*}=
	\begin{pmatrix}
		\disp A+\mu  & \disp -K^{*}B_v^{*} \\
		\disp B^{*}  & \disp -\alpha I-S^{*}K^{*}B_v^{*}  \\
	\end{pmatrix}, \quad D(\mathcal{A}^{*})= D(A)\times L^2(\Gamma_1).
\end{equation}

\begin{lemma}\label{lem3.2}
	In addition to Assumption \ref{assumption 1.1},  suppose that  $\alpha>0$  satisfies
	\begin{equation} \label{mys202202161454}
		-\alpha-\mu\neq \lambda_j, \quad j\in \mathbb{Z}_+.
	\end{equation}
Then, the operator $\mathcal{A}^{*}$  given by \dref{mys01092134} generates an exponentially stable $C_0$-semigroup on $L^2(\Omega)\times L^2(\Gamma_1)$,
where  $B^*$,  $B_v^{*}$, $K^{*}$ and $S^{*}$  are given by \dref{1.1.7}, (\ref{mys01091933}), (\ref{mys01091956}) and (\ref{mys01091957}), respectively.
\end{lemma}

\begin{proof}
Similarly to the proof  in Theorem \ref{thm2.2}, we introduce the  following transformation
\begin{equation} \label{mys01272011}
\mathbb{T}(f,g)^{\top}=(f,g-S^{*}f)^{\top}, \quad\forall\  (f,g)^{\top}\in L^2(\Omega)\times L^2(\Gamma_1),
\end{equation}
where $S^{*}\in\mathcal{L} \big(L^2(\Gamma_1), L^2(\Omega)\big)$ is given by (\ref{mys01091957}).
 By a simple computation, we can conclude that $\mathbb{T}\in L^2(\Omega)\times L^2(\Gamma_1)$ is invertible and its inverse is
\begin{equation}
\mathbb{T}^{-1}(f,g)^{\top}=(f,g+S^{*}f)^{\top}, \qquad \forall\ (f,g)^{\top}\in L^2(\Omega)\times L^2(\Gamma_1).
\end{equation}
Furthermore,  we obtain
\begin{equation}\label{mys01272014}
\mathbb{T}\mathcal{A}^{*}\mathbb{T}^{-1}=\mathcal{A}^{*}_{\mathbb{T}}, \qquad D(\mathcal{A}^{*}_{\mathbb{T}})=\mathbb{T}D(\mathcal{A}^{*}),
\end{equation}
where $\mathcal{A}^{*}_{\mathbb{T}}$ satisfies
\begin{equation} \label{mys01272020}
\disp \mathcal{A}^{*}_{\mathbb{T}}=
\begin{pmatrix}
	\disp A+\mu-K^{*}B_v^{*}S^{*}  & -K^{*}B_v^{*}  \\
	\disp -S^{*}(A+\mu)-\alpha S^{*}+B^{*}  & -\alpha  \\
\end{pmatrix}.
\end{equation}
 Owing to \dref{2.1.10}, we have  $-S^{*}(A+\mu)-\alpha S^{*}+B^{*}=0$ and hence
\begin{equation} \label{20222281542}
\disp \mathcal{A}^{*}_{\mathbb{T}}=
\begin{pmatrix}
	\disp A+\mu-K^{*}B_v^{*}S^{*}  & -K^{*}B_v^{*}  \\
	\disp 0  & -\alpha  \\
\end{pmatrix}.
\end{equation}

Since $A+\mu$ generates an analytic semigroup $e^{(A+\mu)t}$ on $L^2(\Omega)$ and $K^{*}B_v^{*}S^{*}$ is bounded, it follows from \cite[Corollary 2.2, p.81]{Pazy1983} that $A+\mu-K^{*}B_v^{*}S^{*}$ also generates an analytic semigroup on $L^2(\Omega)$. The point spectrum satisfies $ \sigma_p(A+\mu-K^{*}B_v^{*}S^{*})=\sigma_p(A+\mu-SB_vK)\subset\{s\ |\ {\rm Re}\, s <0\}$. Noting that $A+\mu-K^{*}B_v^{*}S^{*}$
is the inverse of a compact operator, it follows from \cite[Theorem4.3, p.118]{Pazy1983} that the operator $A+\mu-K^{*}B_v^{*}S^{*}$ generates an exponentially stable $C_0$-semigroup on $L^2(\Omega)$. Owing to the block-triangle structure and \cite[Lemma 5.1]{Weiss1997TAC}, the operator $\mathcal{A}^{*}_{\mathbb{T}}$ generates an exponentially stable $C_0$-semigroup $e^{\mathcal{A}^{*}_{\mathbb{T}}t}$ on $L^2(\Omega)\times L^2(\Gamma_1)$. Therefore, the operator $\mathcal{A}^{*}$ also generates an exponentially stable $C_0$-semigroup on $L^2(\Omega)\times L^2(\Gamma_1)$ due to the similarity (\ref{mys01272014}).
\end{proof}

\section{Observer design}
Inspired by the dynamic compensation  method in   \cite{Feng2021},
we add, in terms of $\phi_1, \phi_2, \cdots, \phi_N$ given by  (\ref{1.1.5}),  the sensor dynamic to
 system (\ref{1.1}):
\begin{equation} \label{3.2.1}
	\begin{cases} w_{t}(x,t)-\Delta w(x,t)-\mu w(x,t) =0,\qquad (x,t)\in \Omega\times	(0,+\infty),\crr
		w(x,t)=0,\qquad \qquad (x,t)\in \Gamma_0\times(0,+\infty),\crr
	\disp	\frac{\partial w(x,t)}{\partial \nu} = u(x,t), \qquad (x,t)\in \Gamma_1\times(0,+\infty),\crr
		\disp p_{t}(\cdot,t)=-\alpha p(\cdot,t)+B^{*}w(\cdot,t),  \qquad \mbox{in}\ \  \Gamma_1, \crr
		y_p(t)=\big( y_1(t),y_2(t),\cdots,y_N(t) \big)^{\top},
		%\begin{pmatrix}
%			\int_{\Gamma_1} v(x,t)\phi_1(x)dx\\
%			\int_{\Gamma_1} v(x,t)\phi_2(x)dx\\
%			\cdots \\
%			\int_{\Gamma_1} v(x,t)\phi_N(x)dx\\
%		\end{pmatrix},			
		\qquad t\in (0,+\infty),
	\end{cases}
\end{equation}
where $\alpha>0$ is a tuning parameter, $p(\cdot,t)$ is an extended state, $B^{*}$ is given by (\ref{1.1.7}) and
\begin{equation}\label{FH2022122932}
y_i(t)=\int_{\Gamma_1} p(x,t)\phi_i(x)dx,\ \ i=1,2,\cdots,N.
\end{equation}
 By (\ref{1.2}), (\ref{1.1.6.1}), \dref{1.1.7} and (\ref{mys01091933}), system (\ref{3.2.1}) can be written abstractly as
\begin{equation}\label{mys01091942}
	\begin{cases} w_{t}(\cdot,t)=(\tilde{A}+\mu) w(\cdot,t)+Bu(\cdot,t), \crr
		\disp p_{t}(\cdot,t)=-\alpha p(\cdot,t)+ B^{*} w(\cdot,t), \crr
			y_p(t)=	B_{v}^{*} p(\cdot,t).
	\end{cases}
\end{equation}
%\iffalse
%where $C_{\phi_i}: L^2(\Gamma_1) \rightarrow \mathbb{R}$ is defined by
%\begin{equation}\label{3.2.a}
%      C_{\phi_i}h=\int_{\Gamma_1}h(x)\phi_i(x)dx, \qquad \forall h\in L^2(\Gamma_1)
%\end{equation}
%\fi
Inspired by the  method in   \cite{Feng2021}, the observer of system (\ref{mys01091942}) can be designed as
\begin{equation} \label{3.2.5}
	\begin{cases} \hat{w}_{t}(x,t)=\Delta \hat{w}(x,t)+\mu \hat{w}(x,t)
		\disp  + K^{*}B_v^{*}[ p(\cdot,t)- \hat{p}(\cdot,t)],\quad x\in \Omega,\crr
		\hat{w}(x,t)=0,\qquad x\in \Gamma_0,\crr
		\disp \frac{\partial \hat{w}(x,t)}{\partial \nu} = u(x,t), \qquad x\in \Gamma_1,\crr
		\hat{p}_{t}(\cdot,t)=-\alpha  \hat{p}(\cdot,t)+B^{*}\hat{w}(\cdot,t)+S^{*}K^{*}B_v^{*}[p(\cdot,t)-\hat{p}(\cdot,t)] \quad \mbox{in} \ \ \Gamma_1,
	\end{cases}
\end{equation}
where $B_v^{*}$, $K^{*}$ and $S^{*}$ are given by (\ref{mys01091933}), (\ref{mys01091956}) and (\ref{mys01091957}), respectively.
%\iffalse
%\begin{itemize}
%	\item Solve Sylvester equation
%	 \begin{equation} \label{5.x}
%	 	P_i(\tilde{A}+\mu)+\alpha P_i+NQ_{v_i}B^{*}=0.
%	 \end{equation}
%  to get $P_i\in \mathcal{L}(L^2(\Omega),L^2(\Gamma_1))$.
%	\item Design control $K_j\in \mathcal{L}(L^2(\Omega),\mathbb{R})$ such that $\tilde{A}+\mu+\sum_{j=1}^N K_jC_{v}P_j$ is Hurwitz
%\end{itemize}
%
%By a straightforward computaition, the solution of (\ref{5.x}) is found to be
%\begin{equation}\label{5.aa}
%	P_i=-N Q_{v_i}B^{*}(\alpha+\mu+A)^{-1}\in \mathcal{L}(L^2(\Omega),L^2(\Gamma_1))
%\end{equation}
%
%By (\ref{1.1.7}) (\ref{4.1})and (\ref{3.2.2}), we have
%\begin{equation} \label{5.c}
%	C_{v} P_j=N J_{\phi_j}^{\gamma}\in\mathcal{L}(L^2(\Omega),\mathbb{R})\ \ with \ \ \gamma=-\alpha-\mu \quad (j=1,2,\cdots,N)
%\end{equation}
%\vskip 0.5cm
%
%By Lemma \ref{lem3.2}, (\ref{4.9}) and (\ref{2.10}), the operator $K_j$ can be chosen by (\ref{4.10}), where $(k_{ij})_{N\times N}=\mathbb{K}_N$ is a matrix such that $\Lambda_N+\mathbb{K}_N J_N$ is Hurwitz. As a result of (\ref{3.2.2}), (\ref{5.aa}) and (\ref{4.10})
%\begin{equation}\label{3.2.7}
%	P_iK_j=-N\sum_{l=1}^{N}k_{lj}\phi_i(x)\xi_j|_{\Gamma_1} \qquad i,j=1,2,\cdots,N
%\end{equation}
%where $\xi_j$ is given by
%\begin{equation}\label{3.2.8}
%	\begin{cases}
%		(\alpha+\mu+\Delta)\xi_j=\phi_j \qquad x\in{\Omega}\cr
%		\xi_j=0 \qquad x\in \Gamma_0 \cr
%		\frac{\partial \xi_j(x)}{\partial \nu}=0 \qquad x\in \Gamma_1
%	\end{cases}
%\end{equation}
%\fi
Combining (\ref{1.1.7}), (\ref{mys01091933}),  (\ref{mys01091956}), (\ref{mys01091845}),  (\ref{mys01091957})  and (\ref{3.2.1}), the abstract observer (\ref{3.2.5}) can be
 written concretely
\begin{equation} \label{3.2.9}
	\begin{cases} \hat{w}_{t}(x,t)=\Delta \hat{w}(x,t)+\mu \hat{w}(x,t) \crr
		 \disp \qquad \quad \ \ +\sum_{i,j=1}^N l_{ij}\phi_j(x)\int_{\Gamma_1}\phi_i(x)\big(p(x,t)-\hat{p}(x,t)\big)dx ,\qquad x\in \Omega,\crr
		\hat{w}(x,t)=0,\qquad x\in \Gamma_0,\crr
		\disp\frac{\partial \hat{w}(x,t)}{\partial \nu} = u(x,t), \qquad x\in \Gamma_1,\crr
		\disp \hat{p}_{t}(x,t)=-\alpha  \hat{p}(x,t)+\hat{w}(x,t)   \crr
		\disp \qquad \quad \ \ +\sum_{i,j=1}^N l_{ij}\xi_{\phi_j}(x)\int_{\Gamma_1}\phi_i(x)\big(p(x,t)-\hat{p}(x,t)\big)dx,\quad x\in \Gamma_1 ,  \qquad
	\end{cases}
\end{equation}
where $\xi_{\phi_j}(x)$,  satisfies
\begin{equation}\label{mys01092051}
	\begin{cases}
		\disp \Delta \xi_{\phi_j}(x)=(-\alpha-\mu)\xi_{\phi_j}(x)+\phi_j(x),  \ \ x\in \Omega, \crr
		\disp \xi_{\phi_j}(x)=0, \ x\in \Gamma_0; \ \ \frac{\partial \xi_{\phi_j}(x)}{\partial \nu} =0, \ x\in \Gamma_1,
	\end{cases} 	j=1,2,\cdots,N.
\end{equation}

\begin{theorem}
	In addition to Assumption \ref{assumption 1.1}, let $\alpha>0$ satisfy
	\begin{equation} \label{3.2.11}
	-\alpha-\mu\neq \lambda_j, \quad j\in \mathbb{Z}_+.
	\end{equation}
	Then, for any initial value $\big(w(\cdot,0),p(\cdot,0),\hat{w}(\cdot,0),\hat{p}(\cdot,0)\big)^{\top}\in [L^2(\Omega)\times L^2(\Gamma_1)]^2 $ and $\disp u\in L^2_{\rm loc}\big([0,\infty),L^2(\Gamma_1)\big)$, the observer (\ref{3.2.9}) of system (\ref{3.2.1}) admits a unique solution $\disp (\hat{w},\hat{p})^{\top}\in C\left([0,\infty), L^2(\Omega)\times L^2(\Gamma_1)\right)$ such that
	\begin{equation}\label{3.2.12}
		\disp \lim_{t\rightarrow\infty} e^{\omega t}\|w(\cdot,t)-\hat{w}(\cdot,t),p(\cdot,t)-\hat{p}(\cdot,t)\|_{L^2(\Omega)\times L^2(\Gamma_1)}=0,
	\end{equation}
	where $\omega$ is a positive constant that is independent of $t$.
\end{theorem}

\begin{proof}
	For any $\big(w(\cdot,0),p(\cdot,0)\big)^{\top}\in L^2(\Omega)\times L^2(\Gamma_1) $ and $u\in L^2_{\rm loc}\left([0,\infty),L^2(\Gamma_1)\right)$, it is well known that the control plant (\ref{3.2.1}) admits a unique solution $(w,p)^{\top}\in C\big([0,\infty), L^2(\Omega)\times L^2(\Gamma_1)\big)$ such that $y_{j}\in L^2_{\rm loc}[0,\infty)$ for any $ j=1,2,\cdots,N$. Let
	\begin{equation}\label{3.2.13}
		\begin{cases}
			\tilde{w}(x,t)=w(x,t)-\hat{w}(x,t), \quad (x,t)\in \Omega\times[0,\infty), \crr
			\tilde{p}(s,t)=p(s,t)-\hat{p}(s,t),  \qquad (s,t)\in \Gamma_1\times[0,\infty).
		\end{cases}
	\end{equation}
Then the errors are governed by
   \begin{equation} \label{mys01092105}
	\begin{cases}  \tilde{w}_{t}(x,t)=\Delta \tilde{w}(x,t)+\mu \tilde{w}(x,t)- K^{*}B_v^{*} \ \tilde{p}(\cdot,t),\ \ x\in \Omega,\crr
		\disp\tilde{w}(x,t)=0,\qquad x\in \Gamma_0,\crr
		\disp\frac{\partial \tilde{w}(x,t)}{\partial \nu} = 0, \qquad x\in \Gamma_1,\crr
		\disp \tilde{p}_{t}(\cdot,t)=-\alpha  \tilde{p}(\cdot,t)+B^{*}\tilde{w}(\cdot,t)-S^{*}K^{*}B_v^{*} \ \tilde{p}(\cdot,t) \quad \mbox{in} \ \ \Gamma_1.
	\end{cases}
   \end{equation}
In terms of the operator $\mathcal{A}^{*} $ given by \dref{mys01092134}, system (\ref{mys01092105}) can be written as
\begin{equation} \label{3.2.15}
	\frac{d}{dt}\big(\tilde{w}(\cdot,t),\tilde{p}(\cdot,t)\big)^{\top}
=\mathcal{A}^{*}\big(\tilde{w}(\cdot,t),\tilde{p}(\cdot,t)\big)^{\top}.
\end{equation}

By   Lemma \ref{lem3.2},  the operator $\mathcal{A}^{*}$   generates an exponentially stable analytic semigroup  $e^{\mathcal{A}^{*}t}$ on $L^2(\Omega)\times L^2(\Gamma_1)$.
%it follows from \cite[Theorem 4.3, p.118]{Pazy1983}, the point spectrum satisfy  $\sigma_p(\mathcal{A})=\sigma_p(\mathcal{A}^{*}) \subset \{s\in\mathbb{C}|{\rm Re}(s)<0\}$.
%Then the operator $\mathcal{A}^{*}$ also generates an exponentially stable $C_0$-semigroup $e^{\mathcal{A}^{*}t}$ on $L^2(\Omega)\times L^2(\Gamma_1)$.
Hence, the error system \dref{mys01092105} with initial
state $\big(\tilde{w}(\cdot,0),\tilde{p}(\cdot,0)\big)^{\top}\in L^2(\Omega)\times L^2(\Gamma_1)$ admits a unique solution  $\big(\tilde{w}(\cdot,t),\tilde{p}(\cdot,t)\big)^{\top}\in C\big([0,\infty);L^2(\Omega)\times L^2(\Gamma_1)\big)$  such that
\begin{equation}\label{3.2.20}
	\lim_{t\rightarrow\infty} e^{\omega t}\|\tilde{w}(\cdot,t) ,\tilde{p}(\cdot,t)\|_{L^2(\Omega)\times L^2(\Gamma_1)}=0,
\end{equation}
where $\omega$ is a positive constant that is independent of $t$.
Let $\big(\hat{w}(\cdot,t),\hat{p}(\cdot,t)\big)=\big(w(\cdot,t)-\tilde{w}(\cdot,t),p(\cdot,t)-\tilde{p}(\cdot,t)\big)$, it shows that such a
defined $(\hat{w}, \hat{p})^{\top} \in C\big([0,\infty);L^2(\Omega)\times L^2(\Gamma_1)\big)$ is a solution of system (\ref{3.2.9}) or equivalently system \dref{3.2.5}. Moreover, (\ref{3.2.12}) holds due to (\ref{3.2.13}) and (\ref{3.2.20}). Owing to the linearity of system (\ref{3.2.9}), the solution is unique.
\end{proof}

\section{Conclusions}
In this paper, we extend the results in \cite{Feng2021} to the general multi-domain.
By introducing the ODE   actuator/sensor dynamics,
the difficulties caused by instability can be solved
 by  the newly developed dynamics
compensation approach \cite{Feng2008}, \cite{Feng2009}.
Since both the full state feedback
law and the state observer are designed, the observer based output feedback is actually proposed to stabilize exponentially  the unstable multi-dimensional  heat system.
The  dynamics compensation approach may also used to the other multi-dimensional PDEs.
Our future work is the stabilization and observation of the
multi-dimensional
  unstable wave equation.

%
% It is very interesting to apply  dynamics compensation approach   to the other unstable PDE systems.
%Our future works  is
%to stabilize the unstable  Euler-Bernoulli equation.

%we improved the previous method, and our improved approach makes our assumptions hold for the heat equation in the general domain. The improved dynamics compensation method basically solves the instability problem on the general domain for any unstable term $\mu$. Regarding the design of the observer, we reduce some of the previous complicated steps and use the adjoint operator method to prove the stability of the observer.
%
%An related important problemis is to extend this new method to
%other unstable PDEs with non-self-adjoint operator, especially for multi-dimensimonal unstable PDEs. The stabilization of the unstable wave equation is our future work.

\section{Appendix}

\begin{lemma}
	\label{lem6.1.2}
	Let the operator $A$ be given by (\ref{1.2}).
For any $\lambda\in \sigma(A)$, suppose that
$\psi_1,\psi_2,\cdots,\psi_{m_\lambda}$ are the eigenfunctions
  corresponding to
   eigenvalue $\lambda $ of $A$, where $m_{\lambda}$ is the geometric multiplicity of $\lambda$.
   If  $\psi_1,\psi_2,\cdots,\psi_{m_\lambda}$
are linearly independent on $L^2(\Omega)$, then the following matrix   is invertible
\begin{equation} \label{FH20222262127}
   \begin{pmatrix}
             \langle \psi_{1},\psi_{1} \rangle_{L^2(\Gamma_1)}&
             \langle \psi_{1},\psi_{2} \rangle_{L^2(\Gamma_1)}&\cdots&\langle \psi_{1},\psi_{m_{\lambda}} \rangle_{L^2(\Gamma_1)} \\
              \langle \psi_{2},\psi_{1} \rangle_{L^2(\Gamma_1)}&
             \langle \psi_{2},\psi_{2} \rangle_{L^2(\Gamma_1)}&\cdots&\langle \psi_{2},\psi_{m_{\lambda}} \rangle_{L^2(\Gamma_1)} \\
                \vdots&\vdots&\cdots&\vdots\\
                   \langle \psi_{m_{\lambda}},\psi_{1} \rangle_{L^2(\Gamma_1)}&
             \langle \psi_{m_{\lambda}},\psi_{2} \rangle_{L^2(\Gamma_1)}&\cdots&\langle \psi_{m_{\lambda}},\psi_{m_{\lambda}} \rangle_{L^2(\Gamma_1)} \\
                  \end{pmatrix}.
 \end{equation}
 \end{lemma}

\begin{proof}
 Since  the operator $A $    is self-adjoint and negative with compact resolvents,
$\sigma(A)$ consists of isolated eigenvalues of finite
 geometric multiplicity only.
Hence,   $\psi_{k }$ satisfies
	\begin{equation} \label{6.1.5}
		\begin{cases} \disp \Delta\psi_{k }(x) =\lambda  \psi_{k }(x), \quad x\in \Omega,\crr
			\disp \psi_{k }(x)=0,\quad x\in \Gamma_0,\crr
			\disp \frac{\partial \psi_{k }(x) }{\partial \nu}= 0, \quad x\in \Gamma_1,
		\end{cases},\ \ k=1,2,\cdots,m_{\lambda}.
	\end{equation}
Noting that the matrix in \dref{FH20222262127} happens to be a Gram matrix of the sequence $\psi_1,\psi_2,\cdots,\psi_{m_\lambda}$,    the proof will be accomplished if we can prove that
 $\psi_1,\psi_2,\cdots,\psi_{m_\lambda}$
are linearly independent on $L^2(\Gamma_1)$ (\cite[p.441, Theorem 7.2.10]{Horn2013}).

 Actually, suppose    there exist   $ \ell_1,\ell _2,\cdots,\ell _{m_\lambda}\in \R $ such that
	\begin{equation} \label{6.1.4}
	\ell _1\psi_{1}(x)+\ell _{2}\psi_{2}(x)+\cdots+\ell _{{m_\lambda}}\psi_{m_\lambda}(x)=0 , \quad x\in \Gamma_1.
	\end{equation}
If we let
\begin{equation} \label{2022218735}
\beta_{\ell} (x)=	\ell _1\psi_{1}(x)+\ell _{2}\psi_{2}(x)+\cdots+\ell _{m_\lambda}\psi_{m_\lambda}(x) , \quad x\in \Omega,
	\end{equation}
then it follows from \dref{6.1.5} that
\begin{equation} \label{FH6.1.5}
		\begin{cases} \disp \Delta \beta_{\ell}(x) =\lambda  \beta_{\ell}(x), \quad x\in \Omega,\crr
			\disp \beta_{\ell}(x)=0,\quad x\in \Gamma ,\crr
			\disp \frac{\partial \beta_{\ell}(x)}{\partial \nu} =0, \quad x\in \Gamma_1 ,
		\end{cases}
	\end{equation}
 which  admits a unique solution $\beta_{\ell }(x)\equiv0$, $x\in\Omega$  (\cite[Theorem 2.4]{Feng2015}).
 Since   $\psi_1,\psi_2,\cdots,\psi_{m_\lambda}$
are linearly independent on $L^2(\Omega)$,
 we can conclude  that $\ell_1=\ell_2=\cdots=\ell_{m_\lambda}=0$.
  As a result,   $\psi_1,\psi_2,\cdots,\psi_{m_\lambda}$ are linearly independent on $L^2(\Gamma_1)$.
  \end{proof}

\begin{lemma}\label{lem6.x}
Let  $N$  be a positive integer,
  $F_N\in \R^{N\times N}$ and
\begin{equation}\label{FH20222261601}
		\Lambda_{N}=
 {\rm diag}(J_1,\cdots,J_p),\ \
	J_{j}={\rm diag}(\lambda_j,\cdots,\lambda_j)\in\R^{n_j\times n_j} ,\ \ j=1,2,\cdots,p,
	\end{equation}
where $n_1,n_2,\cdots,n_p$  are  $p$ positive integers such  that $n_1+n_2+\cdots+n_p=N$.
	%Let the $\Lambda_N$, $F_N\in \R^{N\times N}$ be given by
 Suppose that
  \begin{equation}\label{FH20222261634}
  \lambda_k\neq\lambda_j \ \mbox{if and only if}\ \  k\neq j, \ \ \ k,j=1,2,\cdots,p
  \end{equation}
  and the matrix  $F_N $ can be written as
 \begin{equation}\label{02252019}
		%\Lambda_{j}={\rm diag}(\lambda_1,\cdots,\lambda_N),\ \
	 	F_{N}=
\begin{pmatrix}
			P_1 & * & \cdots & * \\
			* & P_2  & \cdots & * \\
			\vdots & \vdots  & \ddots & \vdots  \\
			* & *  & \cdots & P_{p}  \\
		\end{pmatrix}  , \ \
P_{j}\in\R^{n_j\times n_j},\ \
		%\begin{pmatrix}
%			f^j_{11} & f^j_{12}  & \cdots & f^j_{1n_j}  \\
%			f^j_{21} & f_{22}  & \cdots & f^j_{2n_j}  \\
%			\vdots & \vdots  & \cdots & \vdots  \\
%			f^j_{n_j1} & f^j_{n_j2}  & \cdots & f^j_{n_jn_j}  \\
%		\end{pmatrix}  ,
j=1,2,\cdots,p.
	\end{equation}
Then,   the pair $(\Lambda_N,F_N)$ is controllable provided  the matrix determinant of   $P_j$ satisfies
\begin{equation}\label{F20221628226}
\left| P_j\right|\neq 0,\ \ \ j=1,2,\cdots,p.
\end{equation}
	%where $N$ is a positive integer. Suppose that  $\lambda_1\leq\lambda_2\leq\cdots\leq \lambda_N$. For any $m_k$ equal eigenvalue $\lambda_k=\lambda_{k+1}=\cdots=\lambda_{k+m_k-1}$, where $1\leq k\leq N$, $k+m_k-1\leq N$, suppose the submatrix $F_{m_k}$ of $F_N$ satisfies
	%\begin{equation}\label{02252144}
%		\left|F_{m_k}\right|=\left|
%		\begin{array}{cccc}
%			f_{k,k} & f_{k,k+1}  & \cdots & f_{k,k+m_k-1}  \\
%			f_{k+1,k} & f_{k+1,k+1}  & \cdots & f_{k+1,k+m_k-1}  \\
%			\vdots & \vdots  & \cdots & \vdots  \\
%			f_{k+m_k-1,k} & f_{k+m_k-1,k+1}  & \cdots & f_{k+m_k-1,k+m_k-1}  \\
%		\end{array}\right| \neq 0 ,
%	\end{equation}
%	then the pair $(\Lambda_N,F_N)$ is controllable.
\end{lemma}

\begin{proof}
Otherwise, we can conclude that $(\Lambda_N^\top,F^{\top}_N)$ is not observable.
   By the Hautus test \cite[p.15, Remark 1.5.2]{Tucsnak 2009},
there exist $0\neq V\in \R^N$ and  $k \in \{ 1,2,\cdots, p\}$ such that
\begin{equation}\label{F20222261650}
   \Lambda_N V=\lambda_k V \ \ \mbox{and}\ \ F_N^{\top}V=0.
   \end{equation}
  Without   loss of   generality,
we suppose that
\begin{equation}\label{F20222261640}
V=\begin{pmatrix}
    V_1 \\
    V_2\\
    \vdots\\
    V_{p}
  \end{pmatrix},\ \ V_j=(v^j_1,\cdots ,v_{n_j}^j)^{\top},\ \ j=1,2,\cdots,n_j.
\end{equation}
The first equation of \dref{F20222261650}  becomes
\begin{equation}\label{F20222261650*}
   \Lambda_N V-\lambda_k V=\begin{pmatrix}
    (\lambda_1-\lambda_k) V_1 \\
    \vdots\\
    (\lambda_k-\lambda_k)V_2\\
    \vdots\\
    (\lambda_p-\lambda_k)V_{p}
  \end{pmatrix} = 0. \ \
   \end{equation}
 By the assumption \dref{FH20222261634}, $V_i=0$ when $i\neq k$. As a result,   $F_N^{\top}V=0$  implies that  $P_k^{ \top}V_k=0$. Owing  to the assumption \dref{F20221628226}, we can conclude that $V_k=0$ and hence  $V=0$. This is contradict to the fact $ V\neq 0$. Therefore,
 the pair $(\Lambda_N,F_N)$ is controllable.
\end{proof}

%\begin{lemma}\label{lem6.1.1}
%	Let the operator $A$ be given by (\ref{1.2})  and let the eigenpairs $\{(\phi_j(\cdot),\lambda_j)\}_{j=1}^{\infty}$ of $A$ satisfy \dref{1.1.5} and  \dref{1.1.4}.
%    For any positive integer $k$ and $m$, suppose that  $\lambda_k=\lambda_{k+1}=\cdots=\lambda_{k+m-1}$. Then corresponding eigenfunctions $\phi_{k}, \phi_{k+1}, \cdots, \phi_{k+m-1}$ satisfy
%    	\begin{equation}
%    	\disp  \left|F_{m}\right|=\left|
%    	\begin{array}{cccc}
%    		\disp \langle \phi_{k},\phi_{k} \rangle_{L^2(\Gamma_1)} &\disp \cdots &\disp \langle \phi_{k},\phi_{k+m-1}\rangle_{L^2(\Gamma_1)} \\
%    		\disp \vdots & \disp \cdots & \disp \vdots\\
%    		\disp \langle\phi_{k+m-1},\phi_{k}\rangle_{L^2(\Gamma_1)} & \disp \cdots &\disp \langle\phi_{k+m-1},\phi_{k+m-1}\rangle_{L^2(\Gamma_1)} \\
%    	\end{array}\right|\neq 0.
%    \end{equation}
%
%\end{lemma}
%
%\begin{proof}
%	According to Lemma \ref{lem6.1.2} in Appendix, the eigenfunctions  $\phi_{k}, \phi_{k+1}, \cdots, \phi_{k+m_k-1}$ are linearly independent on $\Gamma_1$. Suppose there exists a vector $a= (a_1,a_2,\cdots,a_{m})\in \R^{m}$ such that
%	$$a F_{m} a^{\top} = \frac{1}{\theta-\lambda_k}\int_{\Gamma_1} (\sum_{i=1}^{m} a_i\phi_i)^2 dx= 0,$$
%	we can conclude that $a\equiv 0$. It is easy to get that the value of the determinant of $|F_{m} |\neq 0$. 		
%\end{proof}
%

\end{document}